\documentclass[a4paper]{amsart}
\usepackage[active]{srcltx}
\usepackage[all]{xy}
\usepackage{subfig}
\usepackage{hyperref}
\usepackage{mathrsfs}

\usepackage{amssymb,latexsym}
\usepackage{mathrsfs}
\usepackage{graphics}
\usepackage{latexsym}
\usepackage{psfrag}
\usepackage{verbatim}
\usepackage[usenames]{color}
\usepackage{pifont,marvosym}
\usepackage{relsize}

\newcounter{assu}

\theoremstyle{plain}
\newtheorem{lemma}{Lemma}[section]
\newtheorem{theorem}[lemma]{Theorem}
\newtheorem{proposition}[lemma]{Proposition}
\newtheorem{corollary}[lemma]{Corollary}

\theoremstyle{definition}
\newtheorem{assumption}[assu]{Assumption}
\newtheorem{definition}[lemma]{Definition}
\newtheorem{remark}[lemma]{Remark}

\numberwithin{equation}{section}

\newcommand{\norm}{\| \cdot \|}
\newcommand{\R}{\mathbb{R}}

\newcommand{\gr}{\textrm{graph}}

\newcommand{\haus}{\mathcal{H}}
\newcommand{\ve}{\varepsilon}

\newcommand{\erre}{\mathbb{R}}
\newcommand{\enne}{\mathbb{N}}

\newcommand{\f}{\varphi}

\newcommand{\weak}{\rightharpoonup}

\begin{document}

\title{The Monge problem in Wiener space}

\author{Fabio Cavalletti}

\address{SISSA, via Bonomea 265, IT-34136 Trieste (ITALY)}
\date{}

\bibliographystyle{plain}

\begin{abstract}
We address the Monge problem in the abstract Wiener space and we give an existence result
provided both marginal measures are absolutely continuous with respect to the infinite dimensional Gaussian measure $\gamma$. 
\end{abstract}

\date{\today}

\maketitle

\tableofcontents

\section{Introduction}
\label{S:intro}

Let $(X,\| \cdot \|)$ be a separable Banach space, $\gamma \in \mathcal{P}(X)$ be an infinite dimensional Gaussian measure and
$H(\gamma)$ be the corresponding Cameron-Martin space with Hilbertian norm $\|\cdot \|_{H(\gamma)}$. Consider two 
probability measures $\mu, \nu \in \mathcal{P}(X)$.
We will prove the existence of a solution for the following Monge minimization problem 
\begin{equation}\label{E:trasp}
\min_{T: T_{\sharp}\mu = \nu} \int_{X} \| x- T(x) \|_{H(\gamma)} \mu (dx),
\end{equation}
provided $\mu$\ and $\nu$ are both absolutely continuous w.r.t. $\gamma$.

Before giving an overview of the paper, we recall the main results on the Monge problem. 

In the original formulation given by Monge in 1781 the problem was settled in $\erre^{d}$, 
with the cost given by the Euclidean norm and the measures $\mu, \nu$ supposed to be absolutely 
continuous and supported on two disjoint compact sets.
The original problem remained unsolved for a long time. 
In 1978 Sudakov \cite{sudak} claimed to have a solution for any distance cost function induced 
by a norm: an essential ingredient in the proof was that if $\mu \ll \mathcal{L}^{d}$ and $\mathcal{L}^{d}$-a.e. $\erre^{d}$ can be decomposed into 
convex sets of dimension $k$, then
then the conditional probabilities are absolutely continuous with respect to the $\haus^{k}$ measure of the correct dimension. 
But it turns out that when $d>2$, $0<k<d-1$ the property claimed by Sudakov is not true. An example with $d=3$, $k=1$ can be found in \cite{larm}.

The Euclidean case has been correctly solved only during the last decade. L. C. Evans and W. Gangbo in \cite{evagangbo} 
solve the problem under the assumptions that 
$\textrm{spt}\,\mu \cap  \textrm{spt}\,\nu = \emptyset$,  $\mu,\nu \ll \mathcal{L}^{d}$ and their densities are Lipschitz functions with compact support.
The first existence results for general absolutely continuous measures $\mu,\nu$ with compact support is independently obtained by 
L. Caffarelli, M. Feldman and R.J. McCann in \cite{caffafeldmc} and by N. Trudinger and X.J. Wang in \cite{trudiwang}. 
M. Feldman and R.J. McCann \cite{feldcann:mani} extend the results to manifolds with geodesic cost. 
The case of a general norm as cost function on $\erre^{d}$, including also the case with non strictly convex unitary ball, 
is solved first in the particular case of crystalline norm by L. Ambrosio, B. Kirchheim and A. Pratelli in 
\cite{ambprat:crist}, and then in fully generality independently by L. Caravenna in \cite{caravenna:Monge} and by T. Champion and L. De Pascale in 
\cite{champdepasc:Monge}. The Monge minimization problem for non-branching geodesic metric space is studied in \cite{biacava:streconv},
where the existence is proven for spaces satisfying a finite dimensional lower curvature bound. 

\subsection{Overview of the paper}
\label{Ss:over}

The approach to this problem is the one of \cite{biacava:streconv}: assume that there exists a transference plan of finite cost, then we can
\begin{enumerate}
\item reduce the problem to transportation problems along distinct geodesics;
\item show that the disintegration of the marginal $\mu$ on each geodesic is continuous;
\item find a transport map on each geodesic and piece them together.
\end{enumerate}

Indeed, since the cost function is lower semi-continuous, the existence of transference plan of finite cost implies the existence of an optimal transference plan.
This permits to reduce the minimization problem to one dimensional minimization problems. 
There an explicit map can be constructed provided the first marginal measure $\mu$ is continuous (i.e. without atoms), for example 
choose the monotone minimizer of the quadratic cost $|\cdot|^2$. The third point is an application of selection theorems. 

All this strategy has already implemented in fully generality in \cite{biacava:streconv}. 
Therefore to obtain the existence of an optimal transference map we have to show that point (2) of the strategy is fulfilled for $\mu$ and $\nu$
absolute continuous with respect to the infinite dimensional Gaussian measure $\gamma$.

We recall the main steps of the reduction to geodesics.

The geodesics used by a given transference plan $\pi$ to transport mass can be obtained from a set $\Gamma$ on which $\pi$ is concentrated. 
It is well-known that every optimal transference plan is concentrated on a $\norm_{H}$-cyclically monotone set.
Since the considered norm is non-branching, $\Gamma$ 
yields a natural partition $R$ of a subset of the transport set $\mathcal T_e$, i.e. the set of points on the geodesics used by $\pi$: 
defining
\begin{itemize}
\item the set $\mathcal T$ made of inner points of geodesics,
\item the set $a \cup b := \mathcal T_e \setminus \mathcal T$ of initial points $a$ and end points $b$,
\end{itemize}
the cyclical monotonicity of $\Gamma$ implies that the geodesics used by $\pi$ are a partition on $\mathcal T$. 
In general in $a$
there are points from which more than one geodesic starts and in $b$ there are points in which more than one geodesic ends, therefore 
the membership to a geodesic can't be an equivalence relation on the set $a\cup b$.
Take as example the unit circle with $\mu = \delta_0$ and $\nu = \delta_\pi$. 

We note here that $\pi$ gives also a direction along each component of $R$ and w.l.o.g. we can assume that $\mu(b) = \nu(a)=0$.

Even if we have a natural partition $R$ in $\mathcal T$ and $\mu(a) = 0$, we cannot reduce the transport problem to one dimensional problems: a necessary and sufficient condition is that the disintegration of the measure $\mu$ is strongly consistent, which is equivalent to the fact that there exists a $\mu$-measurable quotient map $f : \mathcal T \to \mathcal T$ of the equivalence relation $R$. If this is the case, then
\[
m := f_\sharp \mu, \quad \mu = \int \mu_y m(dy), \quad \mu_y(f^{-1}(y)) = 1,
\]
i.e. the conditional probabilities $\mu_y$ are concentrated on the counterimages $f^{-1}(y)$ (which are single geodesics). 
In our setting the strong consistency of the disintegration of $\mu$
is a consequence of the topological properties of the geodesics of $\|\cdot\|_{H(\gamma)}$ considered as curves in $(X,\|\cdot \|)$.
Finally we obtain the one dimensional problems by 
partitioning $\pi$ w.r.t. the partition $R \times (X \times X)$,
\[
\pi = \int \pi_y m(dy), \quad \nu = \int \nu_y m(dy) \quad \nu_y := (P_2)_\sharp \pi_y,
\]
and considering the one dimensional problems along the geodesic $R(y)$ with marginals $\mu_y$, $\nu_y$ and cost the arc length on the geodesic.

At this point we can study the problem of the regularity of the conditional probabilities $\mu_y$.
A natural operation on sets can be considered: the evolution along the transport set. 
If $A$ is a subset of $\mathcal {T}_{e}$, we denote by $T_t(A)$ the set $T_{t}(\Gamma \cap A \times X)$ where $T_{t}$ is the map from $X \times X$ to $X$ 
that associates to a couple of points its convex combination at time $t$.

It turns out that the fact that $\mu(a) = 0$ and the measures $\mu_y$ are continuous depends on the behave of the function $t \mapsto \gamma(T_t(A))$. 

\begin{theorem}[Proposition \ref{P:puntini} and Proposition \ref{P:nonatoms}]
\label{T:-1}
If for every $A$ with $\mu(A)>0$ there exists a sequence $t_{n} \searrow 0$ and a positive constant $C$ such that 
$\gamma(T_{t_{n}}(A)) \geq C \mu(A)$, then $\mu(a) = 0$ and the conditional probabilities $\mu_y$ and $\nu_{y}$ are continuous.
\end{theorem}

This result implies that the existence of a minimizer of the Monge problem is equivalent to the regularity properties of $t \mapsto \gamma(T_t(A))$.
Hence the problem is reduced to verify that the Gaussian measure $\gamma$ satisfies the assumptions of Theorem \ref{T:-1}. 

Actually this is the key part of the paper. 
Let $\mu = \rho_{1} \gamma$ and $\nu =\rho_{2} \gamma$ and assume that  $\rho_{1}$ and $\rho_{2}$ are bounded.
Then we find suitable $d$-dimensional measures $\mu_{d}, \nu_{d}$, absolute continuous w.r.t. the $d$-dimensional Gaussian measure $\gamma_{d}$, 
converging to $\mu$ and $\nu$ respectively, such that (Theorem \ref{T:stimapprox})  $\gamma_{d}$ verifies $\gamma_{d}(T_{d,t}(A)) \geq C \mu_{d}(A)$ where the evolution now is induced by the transport problem between $\mu_{d}$ and $\nu_{d}$ and the constant $C$ does not depend on the dimension. 
Passing to the limit as $d \nearrow + \infty$, we prove the same property for $\gamma$. 
Hence the existence result is proved for measures with bounded densities.
To obtain the existence result in fully generality we observe that the transport set $\mathcal{T}_{e}$ is a transport set 
also for suitable transport problems between measures satisfying the uniformity condition stated above (Proposition \ref{P:noiniz} and Proposition \ref{P:noatom}).

The assumption that  both $\mu$ and $\nu$ are a.c. with respect to $\gamma$ is fundamental. 
Indeed take as example a diffuse measure $\mu$ and $\nu =\delta_{x}$, then the constant in the evolution estimate induced by the optimal transference plan 
will depend on the dimension and passing to the limit we loose all the informations on the evolution.

\begin{theorem}[Theorem \ref{T:esiste}]\label{T:2} 
Let $\mu,\nu \in \mathcal{P}(X)$ with $\mu,\nu \ll \gamma$. 
Then there exists a solution for the Monge minimization problem \eqref{E:trasp}
\[
\min_{T: T_{\sharp}\mu = \nu} \int \| x- T(x) \|_{H(\gamma)} \mu (dx).
\]
Moreover we can find $T$ invertible.
\end{theorem}

Conditions that ensure the existence of a transference plan of finite transference cost can be found in \cite{feyustu:MKWiener}.

\subsection{Structure of the paper}
\label{Ss:struct}

The paper is organized as follows.

In Section \ref{S:preli}, we recall the basic mathematical results we use. In Section \ref{Ss:univmeas} the fundamentals of projective set theory are listed. In Section \ref{S:disintegrazione} we recall the Disintegration Theorem, using the version of \cite{biacar:cmono}. Next, the basic results of selection principles are in Section \ref{Ss:sele}, and some fundamental results in optimal transportation theory  are in Section \ref{Ss:General Facts}.  

In Section \ref{s:wiener} we recall 
the definition of the abstract Wiener space and of the infinite dimensional Gaussian measure.

Section \ref{S:Optimal} shows, omitting the proof, the construction done in \cite{biacava:streconv} on the Monge problem in a generalized non-branching 
geodesic space $(X,d,d_{L})$ where $d_{L}$ is the distance cost.
Using only the $d_{L}$-cyclical monotonicity of a set $\Gamma$ we can obtain a partial order relation $G \subset X \times X$ as follows: $xGy$ iff there exists $(w,z) \in \Gamma$ and a geodesic $\gamma: [0,1] \to X$, with 
$\gamma(0)=w$, $\gamma(1)=z$, such that $x$, $y$ belongs to $\gamma$ and $\gamma^{-1}(x) \leq \gamma^{-1}(y)$. 
This set $G$ is analytic, and allows to define
\begin{itemize}
\item the transport ray set $R$ \eqref{E:Rray},
\item the transport sets $\mathcal T_e$, $\mathcal T$ (with and without and points) \eqref{E:TR0},
\item the set of initial points $a$ and final points $b$ \eqref{E:endpoint0}.
\end{itemize}
Moreover we show that $R \llcorner_{\mathcal T \times \mathcal T}$ is an equivalence relation, we can assume that the set of final points $b$ can be taken $\mu$-negligible. Since all these results are proved in \cite{biacava:streconv}, here we present just a schematic summary of the construction.
In Section \ref{Ss:Wiener} we show that the Wiener space fits into the general setting and that more regularity can be obtained.
Next, Section \ref{Ss:partition} recalls that the disintegration induced by $R$ on $\mathcal T$ is strongly consistent. 
Using this fact, we can define an order preserving map $g$ which maps our transport problem into a transport problem on $\mathcal S \times \R$, where $\mathcal S$ is a cross section of $R$ (Proposition \ref{P:gammaclass}). Finally we show that under this assumption there exists a transference plan with the same cost of $\pi$ which leaves the common mass $\mu \wedge \nu$ at the same place (note that in general this operation lowers the transference cost).

In Section \ref{S:disin} we prove Theorem \ref{T:-1}. 
Let $T_{t}(x,y) = x(1-t) + yt$, then we define $T_{t}(A)$ as the set $T_{t}(\Gamma \cap A \times X)$ where $\Gamma$ is the $\norm_{H(\gamma)}$-cyclically monotone 
set where the transference plan $\pi$ is concentrated.
We show that under the condition
\begin{equation}\label{E:evol}
\mu(A) > 0 \quad \Longrightarrow \quad \gamma(T_{t}(A)) \geq C \gamma(A \cap \{\rho_{1}>0 \})
\end{equation}
the set of initial points $a$ is $\mu$-negligible (Proposition \ref{P:puntini}). Then, under the same assumption, we prove that the conditional probabilities $\mu_y$ 
and $\nu_{y}$ are continuous (Proposition \ref{P:nonatoms}).

In Section \ref{S:approx} we prove that, choosing $\mu_{d} : = P_{d\,\sharp} \mu$ and $\nu_{d}:= P_{d\,\sharp}\nu$, if 
condition \eqref{E:evol} is verified by $\gamma_{d}$ w.r.t. the evolution induced by the finite dimensional Monge problem between $\mu_{d}$ and $\nu_{d}$,
then condition \eqref{E:evol} passes to the limit provided the constant $C$ is independent on the dimension (Theorem \ref{T:approx}).

Section \ref{S:finite} proves condition \eqref{E:evol} in the finite dimensional case. It follows directly from the proof that
the constant $C$ of condition \eqref{E:evol} depends only on the lower and upper bound of the densities of the marginal measures and is independent on the 
dimension.

In Section \ref{S:solu} we obtain the existence of an optimal transport map. 
Proposition \ref{P:noiniz} proves that the set of initial points is $\mu$-negligible and and the set of final points is $\nu$-negligible.
Proposition \ref{P:noatom} proves that the conditional probabilities $\mu_{y}$ and $\nu_{y}$ are continuous.
Finally Theorem \ref{T:esiste} states the existence results.

We end with a list of notations, Section \ref{S:notation}.

\section{Preliminaries}
\label{S:preli}

In this section we recall some general facts about projective classes, the Disintegration Theorem for measures, measurable selection principles, geodesic spaces and optimal transportation problems.

\subsection{Borel, projective and universally measurable sets}
\label{Ss:univmeas}

The \emph{projective class $\Sigma^1_1(X)$} is the family of subsets $A$ of the Polish space $X$ for which there exists $Y$ Polish and $B \in \mathcal{B}(X \times Y)$ such that $A = P_1(B)$. The \emph{coprojective class $\Pi^1_1(X)$} is the complement in $X$ of the class $\Sigma^1_1(X)$. The class $\Sigma^1_1$ is called \emph{the class of analytic sets}, and $\Pi^1_1$ are the \emph{coanalytic sets}.

The \emph{projective class $\Sigma^1_{n+1}(X)$} is the family of subsets $A$ of the Polish space $X$ for which there exists $Y$ Polish and $B \in \Pi^1_n(X \times Y)$ such that $A = P_1(B)$. The \emph{coprojective class $\Pi^1_{n+1}(X)$} is the complement in $X$ of the class $\Sigma^1_{n+1}$.

If $\Sigma^1_n$, $\Pi^1_n$ are the projective, coprojective pointclasses, then the following holds (Chapter 4 of \cite{Sri:courseborel}):
\begin{enumerate}
\item $\Sigma^1_n$, $\Pi^1_n$ are closed under countable unions, intersections (in particular they are monotone classes);
\item $\Sigma^1_n$ is closed w.r.t. projections, $\Pi^1_n$ is closed w.r.t. coprojections;
\item if $A \in \Sigma^1_n$, then $X \setminus A \in \Pi^1_n$;
\item the \emph{ambiguous class} $\Delta^1_n = \Sigma^1_n \cap \Pi^1_n$ is a $\sigma$-algebra
and $\Sigma^1_n \cup \Pi^1_n \subset \Delta^1_{n+1}$.
\end{enumerate}
We will denote by $\mathcal{A}$ the $\sigma$-algebra generated by $\Sigma^1_1$: clearly $\mathcal{B} = \Delta^1_1 \subset \mathcal{A} \subset \Delta^1_2$.

We recall that a subset of $X$ Polish is \emph{universally measurable} if it belongs to all completed $\sigma$-algebras of all Borel measures on $X$:
it can be proved that every set in $\mathcal{A}$ is universally measurable.
We say that $f:X \to \erre \cup \{\pm \infty\}$ is a \emph{Souslin function} if $f^{-1}(t,+\infty] \in \Sigma^{1}_{1}$.

For the proof of the following lemma see \cite{biacava:streconv}.
\begin{lemma}
\label{L:measuregP}
If $f : X \to Y$ is universally measurable, then $f^{-1}(U)$ is universally measurable if $U$ is.
\end{lemma}

\subsection{Disintegration of measures}
\label{S:disintegrazione}

Given a measurable space $(R, \mathscr{R})$ and a function $r: R \to S$, with $S$ generic set, we can endow $S$ with the \emph{push forward $\sigma$-algebra} $\mathscr{S}$ of $\mathscr{R}$:
\[
Q \in \mathscr{S} \quad \Longleftrightarrow \quad r^{-1}(Q) \in \mathscr{R},
\]
which could be also defined as the biggest $\sigma$-algebra on $S$ such that $r$ is measurable. Moreover given a measure space 
$(R,\mathscr{R},\rho)$, the \emph{push forward measure} $\eta$ is then defined as $\eta := (r_{\sharp}\rho)$.

Consider a probability space $(R, \mathscr{R},\rho)$ and its push forward measure space $(S,\mathscr{S},\eta)$ induced by a map $r$. From the above definition the map $r$ is clearly measurable and inverse measure preserving.

\begin{definition}
\label{defi:dis}
A \emph{disintegration} of $\rho$ \emph{consistent with} $r$ is a map $\rho: \mathscr{R} \times S \to [0,1]$ such that
\begin{enumerate}
\item  $\rho_{s}(\cdot)$ is a probability measure on $(R,\mathscr{R})$ for all $s\in S$,
\item  $\rho_{\cdot}(B)$ is $\eta$-measurable for all $B \in \mathscr{R}$,
\end{enumerate}
and satisfies for all $B \in \mathscr{R}, C \in \mathscr{S}$ the consistency condition
\[
\rho\left(B \cap r^{-1}(C) \right) = \int_{C} \rho_{s}(B) \eta(ds).
\]
A disintegration is \emph{strongly consistent with respect to $r$} if for all $s$ we have $\rho_{s}(r^{-1}(s))=1$.
\end{definition}

The measures $\rho_s$ are called \emph{conditional probabilities}.

We say that a $\sigma$-algebra $\mathcal{H}$ is \emph{essentially countably generated} with respect to a measure $m$ if there exists a countably generated $\sigma$-algebra $\hat{\mathcal{H}}$ such that for all $A \in \mathcal{H}$ there exists $\hat{A} \in \hat{\mathcal{H}}$ such that $m (A \vartriangle \hat{A})=0$.

We recall the following version of the disintegration theorem that can be found on \cite{Fre:measuretheory4}, Section 452 (see \cite{biacar:cmono} for a direct proof).

\begin{theorem}[Disintegration of measures]
\label{T:disintr}
Assume that $(R,\mathscr{R},\rho)$ is a countably generated probability space, $R = \cup_{s \in S}R_{s}$ a partition of R, $r: R \to S$ the quotient map and $\left( S, \mathscr{S},\eta \right)$ the quotient measure space. Then  $\mathscr{S}$ is essentially countably generated w.r.t. $\eta$ and there exists a unique disintegration $s \mapsto \rho_{s}$ in the following sense: if $\rho_{1}, \rho_{2}$ are two consistent disintegration then $\rho_{1,s}(\cdot)=\rho_{2,s}(\cdot)$ for $\eta$-a.e. $s$.

If $\left\{ S_{n}\right\}_{n\in \enne}$ is a family essentially generating  $\mathscr{S}$ define the equivalence relation:
\[
s \sim s' \iff \   \{  s \in S_{n} \iff s'\in S_{n}, \ \forall\, n \in \enne\}.
\]
Denoting with p the quotient map associated to the above equivalence relation and with $(L,\mathscr{L}, \lambda)$ the quotient measure space, the following properties hold:
\begin{itemize}
\item $R_{l}:= \cup_{s\in p^{-1}(l)}R_{s} = (p \circ r)^{-1}(l)$ is $\rho$-measurable and $R = \cup_{l\in L}R_{l}$;
\item the disintegration $\rho = \int_{L}\rho_{l} \lambda(dl)$ satisfies $\rho_{l}(R_{l})=1$, for $\lambda$-a.e. $l$. In particular there exists a 
strongly consistent disintegration w.r.t. $p \circ r$;
\item the disintegration $\rho = \int_{S}\rho_{s} \eta(ds)$ satisfies $\rho_{s}= \rho_{p(s)}$ for $\eta$-a.e. $s$.
\end{itemize}
\end{theorem}

In particular we will use the following corollary.

\begin{corollary}
\label{C:disintegration}
If $(S,\mathscr{S})=(X,\mathcal{B}(X))$ with $X$ Polish space, then the disintegration is strongly consistent.
\end{corollary}

\subsection{Selection principles}
\label{Ss:sele}

Given a multivalued function $F: X \to Y$, $X$, $Y$ metric spaces, the \emph{graph} of $F$ is the set
\begin{equation}
\label{E:graphF}
\textrm{graph}(F) := \big\{ (x,y) : y \in F(x) \big\}.
\end{equation}
The \emph{inverse image} of a set $S\subset Y$ is defined as:
\begin{equation}
\label{E:inverseF}
F^{-1}(S) := \big\{ x \in X\ :\ F(x)\cap S \neq \emptyset \big\}.
\end{equation}
For $F \subset X \times Y$, we denote also the sets
\begin{equation}
\label{E:sectionxx}
F_x := F \cap \{x\} \times Y, \quad F^y := F \cap X \times \{y\}.
\end{equation}
In particular, $F(x) = P_2(\gr(F)_x)$, $F^{-1}(y) = P_1(\gr(F)^y)$. We denote by $F^{-1}$ the graph of the inverse function
\begin{equation}
\label{E:F-1def}
F^{-1} := \big\{ (x,y): (y,x) \in F \big\}.
\end{equation}

We say that $F$ is \emph{$\mathcal{R}$-measurable} if $F^{-1}(B) \in \mathcal{R}$ for all $B$ open. We say that $F$ is \emph{strongly Borel measurable} if inverse images of closed sets are Borel. A multivalued function is called \emph{upper-semicontinuous} if the preimage of every closed set is closed: in particular u.s.c. maps are strongly Borel measurable.

In the following we will not distinguish between a multifunction and its graph. Note that the \emph{domain of $F$} (i.e. the set $P_1(F)$) is in general a subset of $X$. The same convention will be used for functions, in the sense that their domain may be a subset of $X$.

Given $F \subset X \times Y$, a \emph{section $u$ of $F$} is a function from $P_1(F)$ to $Y$ such that $\textrm{graph}(u) \subset F$. We recall the following selection principle, Theorem 5.5.2 of \cite{Sri:courseborel}, page 198.

\begin{theorem}
\label{T:vanneuma}
Let $X$ and $Y$ be Polish spaces, $F \subset X \times Y$ analytic, and $\mathcal{A}$ the $\sigma$-algebra generated by the analytic subsets of X. Then there is an $\mathcal{A}$-measurable section $u : P_1(F) \to Y$ of $F$.
\end{theorem}

A \emph{cross-section of the equivalence relation $E$} is a set $S \subset E$ such that the intersection of $S$ with each equivalence class is a singleton. We recall that a set $A \subset X$ is saturated for the equivalence relation $E \subset X \times X$ if $A = \cup_{x \in A} E(x)$.

The next result is taken from \cite{Sri:courseborel}, Theorem 5.2.1.

\begin{theorem}
\label{T:KRN}
Let $Y$ be a Polish space, $X$ a nonempty set, and $\mathcal{L}$ a $\sigma$-algebra of subset of $X$. 
Every $\mathcal{L}$-measurable, closed value multifunction $F:X \to Y$ admits an $\mathcal{L}$-measurable section.
\end{theorem}

A standard corollary of the above selection principle is that if the disintegration is strongly consistent in a Polish space, then up to a saturated set of negligible measure there exists a Borel cross-section.

In particular, we will use the following corollary.

\begin{corollary}
\label{C:weelsupprr}
Let $F \subset X \times X$ be $\mathcal{A}$-measurable, $X$ Polish, such that $F_x$ is closed and define the equivalence relation $x \sim y \ \Leftrightarrow \ F(x) = F(y)$. Then there exists a $\mathcal{A}$-section $f : P_1(F) \to X$ such that $(x,f(x)) \in F$ and $f(x) = f(y)$ if $x \sim y$.
\end{corollary}

\begin{proof}
For all open sets $G \subset X$, consider the sets $F^{-1}(G) = P_1(F \cap X \times G) \in \mathcal{A}$, and let $\mathcal{R}$ be the $\sigma$-algebra generated by $F^{-1}(G)$. Clearly $\mathcal{R} \subset \mathcal{A}$.

If $x \sim y$, then
\[
x \in F^{-1}(G) \quad \Longleftrightarrow \quad y \in F^{-1}(G),
\]
so that each equivalence class is contained in an atom of $\mathcal{R}$, and moreover by construction $x \mapsto F(x)$ is $\mathcal{R}$-measurable.

We thus conclude by using Theorem \ref{T:KRN} that there exists an $\mathcal{R}$-measurable section $f$: this measurability condition implies that $f$ is constant on atoms, in particular on equivalence classes.
\end{proof}

\subsection{General facts about optimal transportation}
\label{Ss:General Facts}

Let $(X,\mathcal B,\mu)$ and $(Y,\mathcal B,\nu)$ be two Polish probability spaces and  $c : X \times Y \to \R$ be a Borel measurable function. Consider the set of \emph{transference plans}
\[
\Pi (\mu,\nu) := \Big\{ \pi \in \mathcal{P}(X\times Y) : (P_1)_\sharp \pi = \mu, (P_2)_\sharp \pi = \nu \Big\}.
\]
Define the functional
\begin{equation}
\label{E:Ifunct}
\begin{array}{ccccl} 
\mathcal{I} &:& \Pi(\mu,\nu) &\to& \erre^{+} \cr
&& \pi &\mapsto& \mathcal{I}(\pi):=\int c \pi.
\end{array}
\end{equation}
The \emph{Monge-Kantorovich minimization problem} is to find the minimum of $\mathcal{I}$ over all transference plans.

If we consider a $\mu$-measurable \emph{transport map} $T : X \to Y$ such that $T_{\sharp}\mu=\nu$, the functional \eqref{E:Ifunct} becomes
\[
\mathcal I(T):= \mathcal I \big( (Id \times T)_\sharp \mu \big) = \int c(x,T(x)) \mu(dx).
\]
The minimum problem over all $T$ is called \emph{Monge minimization problem}.

The Kantorovich problem admits a (pre) dual formulation.

\begin{definition}
A map $\varphi : X \to \erre \cup \{-\infty\} $ is said to be \emph{$c$-concave} if it is not identically $-\infty$ and there exists $\psi : Y \to \erre \cup \{-\infty\}$, $\psi \not\equiv -\infty$, such that
\[
\varphi(x) = \inf_{y \in Y} \big\{ c(x,y) - \psi(y) \big\}.
\]
The \emph{$c$-transform} of $\varphi$ is the function
\begin{equation}
\label{E:ctransf}
\varphi^c(y) := \inf_{x\in X}  \left\{ c(x,y) - \varphi (x) \right\}.
\end{equation}
The \emph{$c$-superdifferential $\partial^c \f$} of $\varphi$ is the subset of $X \times Y$ defined by
\begin{equation}
\label{E:csudiff}
\partial^{c}\f := \Big\{ (x,y) : c(x,y) - \f(x) \leq c(z,y) - \f(z) \ \forall z \in X \Big\} \subset X \times Y.
\end{equation}
\end{definition}

\begin{definition}\label{D:cicl}
A set $\Gamma \subset X \times Y$ is said to be \emph{$c$-cyclically monotone} if, for any $n \in \mathbb{N}$ and for any family $(x_0,y_0),\dots,(x_n,y_n)$ of points of $\Gamma$, the following inequality holds:
\[
\sum_{i=0}^nc(x_i,y_i) \leq \sum_{i=0}^nc(x_{i+1},y_i),
\]
where $x_{n+1} = x_0$.

A transference plan is said to be \emph{$c$-cyclically monotone} if it is concentrated on a $c$-cyclically monotone set.
\end{definition}

Consider the set
\begin{equation}
\label{E:Phicset}
\Phi_c := \Big\{ (\varphi,\psi) \in L^1(\mu) \times L^1(\nu): \varphi(x) + \psi(y) \leq c(x,y) \Big\}.
\end{equation}
Define for all $(\varphi,\psi)\in \Phi_c$ the functional
\begin{equation}
\label{E:Jfunct}
J(\varphi,\psi) := \int \varphi \mu + \int \psi \nu.
\end{equation}

The following is a well known result (see Theorem 5.10 of \cite{villa:Oldnew}).

\begin{theorem}[Kantorovich Duality]
\label{T:kanto}
Let X and Y be Polish spaces, let $\mu \in \mathcal{P}(X)$ and $\nu \in \mathcal{P}(Y)$, and let $c : X \times Y \to [0,+\infty]$ be lower semicontinuous. Then the following holds:
\begin{enumerate}
\item Kantorovich duality:
\[
\inf_{\pi \in \Pi(\mu,\nu)} \mathcal{I} (\pi) = \sup _{(\varphi,\psi)\in \Phi_{c}} J(\varphi,\psi).
\]
Moreover, the infimum on the left-hand side is attained and the right-hand side is also equal to
\[
\sup _{(\varphi,\psi)\in \Phi_{c}\cap C_{b}} J(\varphi,\psi),
\]
where $C_{b}= C_b(X, \erre) \times C_b(Y,\erre)$.
\item If $c$ is real valued and the optimal cost is finite, then there is a measurable $c$-cyclically monotone set $\Gamma \subset X\times Y$, closed if $c$ is continuous, such that for any $\pi \in \Pi(\mu,\nu)$ the following statements are equivalent:
\begin{enumerate}
\item $\pi$ is optimal;
\item $\pi$ is $c$-cyclically monotone;
\item $\pi$ is concentrated on $\Gamma$;
\item there exists a $c$-concave function $\f$ such that $\pi$-a.s. $\f(x)+\f^{c}(y)=c(x,y)$.
\end{enumerate}

\item If moreover
\[
c(x,y) \leq c_{X}(x) + c_{Y}(y), \quad \ c_{X}\ \mu\textrm{-integrable}, \ c_{Y}\ \nu\textrm{-integrable},
\]
then the supremum is attained:
\[
\sup_{\Phi_c} J = J(\f, \f^{c}) = \inf_{\pi \in \Pi(\mu,\nu)} \mathcal{I}(\pi).
\]
\end{enumerate}
\end{theorem}

We recall also that if $-c$ is Souslin, then every optimal transference plan $\pi$ is concentrated on a $c$-cyclically monotone set \cite{biacar:cmono}.

\subsection{Approximate differentiability of transport maps}
The following results are taken from \cite{ambgiglsav:gradient} where are presented in fully generality.

\begin{definition}[Approximate limit and approximate differential]
Let $\Omega \subset \erre^{d}$ be an open set and $f:\Omega \to \erre^{m}$. We say that $f$ has an approximate limit (respectively, approximate differential) at $x \in \Omega$
if there exists a function $g:\Omega \to \erre^{m}$ continuous (resp. differentiable) at $x$ such that the set $\{ f \neq g\}$ has Lebesgue-density 0 at $x$.
In this case the approximate limit (resp. approximate differential) will be denoted by $\tilde f (x)$ (resp. $\tilde \nabla f(x)$).
\end{definition}

Recall that if $f:\Omega \to \erre^{m}$ is $\mathcal{L}^{d}$-measurable, then it has approximate limit $\tilde f (x)$ at $\mathcal{L}^{d}$-a.e. $x \in \Omega$
and $f(x) = \tilde f (x)$ $\mathcal{L}^{d}$-a.e.. 

Consider $m=d$ and denote with $\Sigma_{f}$ the Borel set of points where $f$ is approximate differential. 

\begin{lemma}[Density of the push-forward]\label{L:density}
Let $\rho \in L^{1}(\erre^{d})$ be a nonnegative function and assume that there exists a Borel set $\Sigma \subset \Sigma_{f}$ 
such that $\tilde f \llcorner_{\Sigma}$ is injective and $\{ \rho >0 \} \setminus \Sigma$ is $\mathcal{L}^{d}$-negligible.
Then $f_{\sharp} \rho \mathcal{L}^{d} \ll \mathcal{L}^{d}$ if and only if $|\det \tilde \nabla f|>0$ for $\mathcal{L}^{d}$-a.e. on $\Sigma$ and 
in this case
\begin{equation}\label{E:area}
f_{\sharp} (\rho \mathcal{L}^{d}) = \frac{\rho}{|\det \tilde \nabla f|} \circ \tilde f^{-1}\llcorner_{f(\Sigma)} \mathcal{L}^{d}.
\end{equation}
\end{lemma}

To conclude we include a regularity result for the Monge minimization problem in $\erre^{d}$ with cost $c_{p}(x,y) = |x-y|^{p}$, $p>1$ 
(Theorem 6.2.7 of \cite{ambgiglsav:gradient}): 
\begin{equation}\label{E:monge}
\min_{T: T_{\sharp}\mu = \nu} \int_{\erre^{d}}  c_{p}(x,T(x)) \mu(dx).
\end{equation}
\begin{theorem}\label{T:different}
Assume that $\mu \in \mathcal{P}^{r}(\erre^{d})$, $\nu \in \mathcal{P}(\erre^{d})$ and
\[
\mu\bigg( \Big\{ x\in \erre^{d} : \int c_{p}(x,y) \nu (dy)< +\infty  \Big\}\bigg), 
\nu\Big( \Big\{ y\in \erre^{d} : \int c_{p}(x,y) \mu (dx) < +\infty \Big\}\bigg) >0. 
\]
If the minimum of \eqref{E:Ifunct} is finite, then 
\begin{itemize}
\item[$i)$] there exists a unique solution $T_{p}$ for the Monge problem \eqref{E:monge}; 
\item[$ii)$] for $\mu$-a.e. $x \in \erre^{d}$ the map $T_{p}$ is approximately differentiable at $x$ and $\tilde \nabla T_{p}(x)$ 
is diagonalizable with nonnegative eigenvalues.
\end{itemize}
\end{theorem}

\section{The Abstract Wiener space}
\label{s:wiener}

In this section we describe our setting. The main reference is \cite{boga:gauss}. 

Given an infinite dimensional separable Banach space $X$,
we denote by  $\|\cdot\|_{X}$ its norm and $X^{*}$ denotes the topological dual, with duality $\langle\cdot, \cdot\rangle$. 
Given the elements $x^{*}_{1}, \dots, x^{*}_{m}$ in $X^{*}$, we denote by $\Pi_{x^{*}_{1}, \dots, x^{*}_{m}} : X \to \erre^{m}$ the map
\[
\Pi_{x^{*}_{1}, \dots, x^{*}_{m}}( x) : = \left( \langle x, x^{*}_{1}\rangle, . . . , \langle x, x^{*}_{m}\rangle\right).
\] 
Denoted with $\mathcal{E}(X)$ the $\sigma$-algebra generated by $X^{*}$. 
A set $C \in \mathcal{E}(X)$ is called a \emph{cylindrical set} and if 
\[
C = \{ x \in X : \Pi_{\{x^{*}_{i}\}}  \in B \}, \quad B \subset \erre^{\infty}, \quad \{x^{*}_{i}\}_{i\in \enne} \subset X^{*} 
\]
we will denote the cylindrical sets with $C(B)$, and $B$ is the base of $C$. In our setting $\mathcal{B}(X)= \mathcal{E}(X)$.

Let $\gamma$ be a non-degenerate centred Gaussian measure defined on $X$. This means that $\gamma \in \mathcal{P}(X)$, is not concentrated on 
a proper subspace of $X$ and for every 
$x^{*} \in X^{*}$ the measure $x^{*}_{\sharp} \gamma$ is a centred Gaussian measure on $\erre$, that is, the Fourier transform of $\gamma$
is given by 
\[
\hat \gamma (x^{*}) = \int_{X} \exp\{ i \langle x^{*}, x \rangle \} \gamma(dx) = \exp \Big\{ -\frac{1}{2} \langle x^{*}, Q x^{*} \rangle   \Big\}
\]
where $Q \in L(X^{*},X)$ is the covariance operator. The non-degeneracy hypothesis of $\gamma$ is equivalent to $\langle x^{*},Qx^{*} \rangle > 0$ 
for every $x^{*} \neq 0$. The covariance operator $Q$ is symmetric, positive and uniquely determined by the relation 
\[
\langle y^{*}, Q x^{*} \rangle = \int_{X} \langle x^{*},x \rangle \langle y^{*} , x \rangle  \gamma(dx), \quad \forall x^{*},y^{*}\in X^{*}.
\]
The fact that $Q$ is bounded follows from the Fernique's Theorem, see \cite{boga:gauss}. This imply that any $x^{*} \in X^{*}$ defines a 
function $x \mapsto x^{*}(x)$ that belongs to $L^{p}(X,\gamma)$ for all $p\geq 1$. In particular let us denote by 
$R^{*}_{\gamma} : X^{*} \to L^{2}(X,\gamma)$ the embedding $R^{*}_{\gamma}x^{*}(x): = \langle x^{*}, x \rangle$. 
The space $\mathscr{H}$ given by the closure of $R^{*}_{\gamma}X^{*}$ in $L^{2}(X,\gamma)$ is called the \emph{reproducing kernel} of the Gaussian measure.
The definition is motivated by the fact that if we consider the operator $R_{\gamma} : \mathscr{H} \to X$ whose adjoint is $R_{\gamma}^{*}$ then 
$Q = R_{\gamma}R_{\gamma}^{*}$: 
\[
\langle y^{*}, R_{\gamma}R_{\gamma}^{*}x^{*}  \rangle = \langle R_{\gamma}^{*}y^{*},R_{\gamma}^{*}x^{*} \rangle_{\mathscr{H}} = 
\int_{X} \langle x^{*},x \rangle \langle y^{*} , x \rangle  \gamma(dx) = \langle y^{*}, Q x^{*} \rangle.
\]
It can proven that $R_{\gamma}$ is injective, compact and 
\begin{equation}\label{E:coord}
R_{\gamma}\hat h = \int_{X} \hat h(x) x \gamma(dx), \quad \hat h \in \mathscr{H},
\end{equation}
where the integral is understood in the Bochner or Pettis sense.

The space $H(\gamma) = R_{\gamma} \mathscr{H} \subset X$ is called the Cameron-Martin space. It is a separable Hilbert space with inner product inherited 
from $L^{2}(X,\gamma)$ via $R_{\gamma}$: 
\[
\langle h_{1} , h_{2}  \rangle_{H(\gamma)} = \langle \hat h_{1} , \hat h_{2}  \rangle_{\mathscr{H}}.
\]
for all $h_{1},h_{2} \in H$ with $h_{i} = R_{\gamma} \hat h_{i}$ for $i = 1,2$. Moreover $H$ is a dense subspace of $X$ and by the compactness of $R_{\gamma}$ follows that the embedding of $(H(\gamma),\| \cdot \|_{H(\gamma)})$ into $(X,\| \cdot \|)$ is compact. Note that if $X$ is infinite dimensional then $\gamma(H)=0$
and if $X$ is finite dimensional then $X=H(\gamma)$. 

\subsection{Finite dimensional approximations}\label{Ss:approx}

Using the embedding of $X^{*}$ in $L^{2}(X,\gamma)$ we say that a family $\{ x^{*}_{i} \} \subset X^{*}$ is orthonormal if the corresponding
family $\{ R_{\gamma}^{*} x^{*}_{i}\}$ is orthonormal in $\mathscr{H}$. 
In particular starting from a sequence $\{y^{*}_{i} \}_{i\in \enne} $ whose image under $R_{\gamma}^{*}$ is dense
in $\mathscr{H}$, we can obtain an orthonormal basis $R_{\gamma}^{*}x^{*}_{i}$ of $\mathscr{H}$. Therefore also 
$h_{j} = R_{\gamma}R_{\gamma}^{*}x^{*}_{j}$ provide an orthonormal basis in $H(\gamma)$.  

In the following we will consider a fixed orthonormal basis $\{e_{i}\}$ of $H(\gamma)$ with $e_{i} = R_{\gamma} \hat e_{i}$ for $\hat e_{i} \in R_{\gamma}^{*}X^{*}$.

\begin{proposition}[\cite{boga:gauss}, Proposition 3.8.12]\label{P:appro}
Let $\gamma$ be a centred Gaussian measure on a Banach space $X$ and $\{e_{i}\}$ an orthonormal basis in $H(\gamma)$.
Define $P_{d} x : = \sum_{i=1}^{d} \langle \hat e_{i}, x\rangle e_{i}$. Then the sequence of measures $\gamma_{d}:= P_{d\,\sharp}\gamma$ 
converges weakly to $\gamma$. 
\end{proposition}
The measure $\gamma_{d}$ defined above is a centred non-degenerate $d$-dimensional Gaussian measure
and, due to the orthonormality of $\{e_{i}\}_{i \in \enne}$, with identity covariance matrix.
Note that from \eqref{E:coord} it follows that $\langle \hat e_{j} , x \rangle = \langle e_{i} ,x \rangle_{H}$  for all $x \in H$. Hence
we will not specify whether the measures $\gamma_{d}$ is probability measures on $\erre^{d}$ or on $P_{d}H$:
\[
\gamma_{d} = \hat e_{1\,\sharp} \gamma \otimes \dots \otimes \hat e_{d\,\sharp} \gamma, \qquad
\hat e_{j\,\sharp} \gamma = \frac{1}{\sqrt{2\pi}} \exp\Big\{ -  \frac{x^{2}}{2} \Big\} \mathcal{L}^{1}.
\]
For every $d \in \enne$ we can disintegrate $\gamma$ w.r.t. the partition induced by the saturated sets of $P_{d}$: 
\begin{equation}\label{E:disintg}
\gamma = \int \gamma_{y,d}^{\perp} \gamma_{d}(dy), \quad \gamma_{y,d}^{\perp}(P_{d}^{-1}(y))=1 \quad \textrm{for}\ \gamma_{d}-\textrm{a.e.}\ y.
\end{equation}

\section{Optimal transportation in geodesic spaces}
\label{S:Optimal}
In what follows $(X,d,d_{L})$ is a generalized non-branching geodesic space in the sense of \cite{biacava:streconv}. 
In this Section we retrace, omitting the proof, the construction done in \cite{biacava:streconv} that permits to reduce the Monge problem 
with non-branching geodesic distance cost $d_{L}$, to a family of one dimensional transportation problems.
The triple $(X,\| \cdot \|,\| \cdot \|_{H(\gamma)})$ is a generalized non-branching geodesic space in the sense of \cite{biacava:streconv}.

Using only the $d_L$-cyclical monotonicity of $\Gamma$,  we obtain a partial order relation $G \subset X \times X$.
The set $G$ is analytic, and allows to define the transport ray set $R$, the transport sets $\mathcal T_e$, $\mathcal T$,
and the set of initial points $a$ and final points $b$.
Moreover we show that $R \llcorner_{\mathcal T \times \mathcal T}$ is an equivalence relation and that we can assume the set of final points $b$ to be $\mu$-negligible.

Let $\mu, \nu \in \mathcal{P}(X)$ and let $\pi \in \Pi(\mu,\nu)$ be a $d_L$-cyclically monotone transference plan with finite cost. 
By inner regularity, we can assume that the optimal transference plan is concentrated on a $\sigma$-compact $d_L$-cyclically monotone set 
$\Gamma \subset \{d_L(x,y) < +\infty\}$. By Lusin Theorem, we can require also that $d_L \llcorner_{\Gamma}$ is $\sigma$-continuous:
\begin{equation}\label{E:gamman}
\Gamma = \cup_n \Gamma_n, \ \Gamma_n \subset \Gamma_{n+1} \ \text{compact}, \quad d_L \llcorner_{\Gamma_n} \ \text{continuous.}
\end{equation}

Consider the set
\begin{align}
\label{E:gGamma}
\Gamma' :=&~ \bigg\{ (x,y) : \exists I \in \enne_0, (w_i,z_i) \in \Gamma \ \text{for} \ i = 0,\dots,I, \ z_I = y \crcr
&~ \qquad \qquad w_{I+1} = w_0 = x, \ \sum_{i=0}^I d_L(w_{i+1},z_i) - d_L(w_i,z_i) = 0 \bigg\}.
\end{align}
In other words, we concatenate points $(x,z), (w,y) \in \Gamma$ if they are initial and final point of a cycle with total cost $0$.
One can prove that $\Gamma \subset \Gamma' \subset \{d_L(x,y) < +\infty\}$, if $\Gamma$ is analytic so is $\Gamma'$
and if $\Gamma$ is $d_L$-cyclically monotone so is $\Gamma'$.

\begin{definition}[Transport rays]
\label{D:Gray}
Define the \emph{set of oriented transport rays}
\begin{equation}
\label{E:trG}
G := \Big\{ (x,y): \exists (w,z) \in \Gamma', d_L(w,x) + d_L(x,y) + d_L(y,z) = d_L(w,z) \Big\}.
\end{equation}

For $x \in X$, the \emph{outgoing transport rays from $x$} is the set $G(x)$ and the \emph{incoming transport rays in $x$} is the set $G^{-1}(x)$. Define the \emph{set of transport rays} as the set
\begin{equation}
\label{E:Rray}
R := G \cup G^{-1}.
\end{equation}
\end{definition}
It is fairly easy to prove that $G$ is still $d_{L}$-cyclically monotone, $\Gamma' \subset G \subset \{d_L(x,y) < +\infty\}$ and 
$G$ and $R$ are analytic sets.

\begin{definition} Define the \emph{transport sets}
\begin{subequations}
\label{E:TR0}
\begin{align}
\label{E:TR}
\mathcal T :=&~ P_1 \big( \textrm{graph}(G^{-1}) \setminus \{x = y\} \big) \cap P_1 \big( \textrm{graph}(G) \setminus \{x = y\} \big), \\
\label{E:TRe}
\mathcal T_e :=&~ P_1 \big( \textrm{graph}(G^{-1}) \setminus \{x = y\} \big) \cup P_1 \big( \textrm{graph}(G) \setminus \{x = y\} \big).
\end{align}
\end{subequations}
\end{definition}

From the definition of $G$ one can prove that $\mathcal{T}$, $\mathcal{T}_e$ are analytic sets. 
The subscript $e$ refers to the endpoints of the geodesics: we have
\begin{equation}
\label{E:RTedef}
\mathcal{T}_e = P_1(R \setminus \{x = y\}).
\end{equation}

It follows that we have only to study the Monge problem in $\mathcal{T}_e$: $\pi(\mathcal{T}_e \times \mathcal{T}_e \cup \{x = y\}) = 1$.
As a consequence, $\mu(\mathcal{T}_e) = \nu(\mathcal{T}_e)$ and any 
maps $T$ such that for $\nu \llcorner_{\mathcal{T}_e} = T_\sharp \mu \llcorner_{\mathcal{T}_e}$ can be extended to a 
map $T'$ such that $\nu = T_\sharp \mu$ with the same cost by setting
\begin{equation}
\label{E:extere}
T'(x) =
\begin{cases}
T(x) & x \in \mathcal{T}_e \crcr
x & x \notin \mathcal{T}_e
\end{cases}
\end{equation}

By the non-branching assumption, if $x \in \mathcal{T}$, then $R(x)$ is a single geodesic and therefore
the set $R \cap \mathcal{T} \times \mathcal{T}$ is an equivalence relation on $\mathcal{T}$ that we will call ray equivalence relation. 
Notice that the set $G$ is a partial order relation on $\mathcal{T}_e$.

The next step is to study the set $\mathcal{T}_{e}\setminus \mathcal{T}$.

\begin{definition}
\label{D:endpoint}
Define the multivalued \emph{endpoint graphs} by:
\begin{subequations}
\label{E:endpoint0}
\begin{align}
\label{E:endpointa}
a :=&~ \big\{ (x,y) \in G^{-1}: G^{-1}(y) \setminus \{y\} = \emptyset \big\}, \\
\label{E:endpointb}
b :=&~ \big\{ (x,y) \in G: G(y) \setminus \{y\} = \emptyset \big\}.
\end{align}
\end{subequations}
We call $P_2(a)$ the set of \emph{initial points} and $P_2(b)$ the set of \emph{final points}.
\end{definition}

Even if $a$, $b$ are not in the analytic class, still they belong to the $\sigma$-algebra $\mathcal{A}$.

\begin{proposition}
The following holds:
\begin{enumerate}
\item the sets
\[
a,b \subset X \times X, \quad a(A), b(A) \subset X,
\]
belong to the $\mathcal{A}$-class if $A$ analytic;
\item $a \cap b \cap \mathcal{T}_e \times X = \emptyset$;
\item $a(x)$, $b(x)$ are singleton or empty when $x \in \mathcal{T}$;
\item $a(\mathcal{T}) = a(\mathcal{T}_e)$, $b(\mathcal{T}) = b(\mathcal{T}_e)$;
\item $\mathcal{T}_e = \mathcal{T} \cup a(\mathcal{T}) \cup b(\mathcal{T})$, $\mathcal{T} \cap (a(\mathcal{T}) \cup b(\mathcal{T})) = \emptyset$.
\end{enumerate}
\end{proposition}

Finally we can assume that the $\mu$-measure of final points and the $\nu$-measure of the initial points are $0$:
indeed since the sets $G \cap b(\mathcal{T}) \times X$, $G \cap X \times a(\mathcal{T})$ is a subset of the graph of the identity map,
it follows that from the definition of $b$ one has that
\[
x \in b(\mathcal{T}) \quad \Longrightarrow \quad G(x) \setminus \{x\} = \emptyset,
\]
A similar computation holds for $a$. Hence we conclude that
\[
\pi (b(\mathcal{T}) \times X) = \pi(G \cap b(\mathcal{T}) \times X) = \pi(\{x = y\}),
\]
and following \eqref{E:extere} we can assume that
\[
\mu(b(\mathcal{T})) = \nu(a(\mathcal{T})) = 0.
\]

\subsection{The Wiener case}
\label{Ss:Wiener}

For the abstract Wiener space, it is possible to obtain more regularity for the sets introduced so far.
Let  $d= \norm$ and $d_{L} = \norm_{H}$: by the compactness of the embedding $R_{\gamma}$ of $H$ into X it follows that
\begin{itemize}
\item[(1)] $d_L : X \times X \to [0,+\infty]$ l.s.c. distance;
\item[(2)] $d_L(x,y) \geq C d(x,y)$ for some positive constant $C$;
\item[(3)] $\cup_{x \in K_1, y \in K_2} \gamma_{[x,y]}$ is $d$-compact if $K_1$, $K_2$ are $d$-compact, $d_L \llcorner_{K_1 \times K_2}$ uniformly bounded.
\end{itemize}

The set $\Gamma'$ is $\sigma$-compact: in fact, if one restrict to each $\Gamma_n$
given by \eqref{E:gamman}, then the set of cycles of order $I$ is compact, and thus
\begin{align*}
\Gamma'_{n,\bar I} :=&~ \bigg\{ (x,y) : \exists I \in \{0,\dots,\bar I\}, (w_i,z_i) \in \Gamma_n \ \text{for} \ i = 0,\dots,I, \ z_I = y \crcr
&~ \qquad \qquad w_{I+1} = w_0 = x, \ \sum_{i=0}^I d_L(w_{i+1},z_i) - d_L(w_i,z_i) = 0 \bigg\}
\end{align*}
is compact. Finally $\Gamma' = \cup_{n,I} \Gamma'_{n,I}$.

Moreover, $d_L \llcorner_{\Gamma'_{n,I}}$ is continuous. If $(x_n,y_n) \to (x,y)$, then from the l.s.c. and
\[
\sum_{i=0}^I d_L(w_{n,i+1},z_{n,i}) = \sum_{i=0}^I d_L(w_{n,i},z_{n,i}), \quad w_{n,I+1} = w_{n,0} = x_{n}, \ z_{n,I} = y_{n},
\]
it follows also that each $d_L(w_{n,i+1},z_{n,i})$ is continuous.

Similarly the sets $G$, $R$, $a$, $b$ are $\sigma$-compact: assumption (3) and the above computation in fact shows that
\[
G_{n,I} := \Big\{ (x,y): \exists (w,z) \in \Gamma'_{n,I}, d_L(w,x) + d_L(x,y) + d_L(y,z) = d_L(w,z) \Big\}
\]
is compact. For $a$, $b$, one uses the fact that projection of $\sigma$-compact sets is $\sigma$-compact.

So we have that $\Gamma$, $\Gamma'$, $G$, $G^{-1}$, $a$ and $b$ are $\sigma$-compact sets.

\subsection{Strongly consistency of disintegrations}\label{Ss:partition}
We recall the main results of \cite{biacava:streconv} that permit to define an order preserving map $g$ which maps our transport problem into a transport problem on $\mathcal S \times \R$, where $\mathcal S$ is a cross section of $R$.

The strong consistency of the disintegration follows from the next result.
\begin{proposition}
\label{P:sicogrF}
There exists a $\mu$-measurable cross section $f : \mathcal{T} \to \mathcal{T}$ for the ray equivalence relation $R$. 
\end{proposition}

Up to a $\mu$-negligible saturated set $\mathcal{T}_N$, we can assume it to have $\sigma$-compact range: just let $S \subset f(\mathcal{T})$ be a $\sigma$-compact set where 
$f_\sharp \mu \llcorner_{\mathcal{T}}$ is concentrated, and set
\begin{equation}
\label{E:TNngel}
\mathcal{T}_S := R^{-1}(S) \cap \mathcal{T}, \quad \mathcal{T}_N := \mathcal{T} \setminus \mathcal{T}_S, \quad \mu(\mathcal{T}_N) = 0.
\end{equation}

Having the $\mu\llcorner_{\mathcal{T}}$-measurable cross-section
\[
\mathcal{S} := f(\mathcal{T})= S \cup f(\mathcal{T}_N) = (\textrm{Borel}) \cup (f(\text{$\mu$-negligible})),
\]
we can define the parametrization of $\mathcal{T}$ and $\mathcal{T}_e$ by geodesics.

Using the quotient map $f$, we obtain a unitary speed parametrization of the transport set.
\begin{definition}[Ray map]
\label{D:mongemap}
Define the \emph{ray map $g$} by the formula
\begin{align*}
g :=&~ \Big\{ (y,t,x): y \in \mathcal{S}, t \in [0,+\infty), x \in G(y) \cap \{d_L(x,y) = t\} \Big\} \crcr
&~ \cup \Big\{ (y,t,x): y \in \mathcal{S}, t \in (-\infty,0), x \in G^{-1}(y) \cap \{d_L(x,y) = -t\} \Big\} \crcr
=&~ g^+ \cup g^-.
\end{align*}
\end{definition}

\begin{proposition}
\label{P:gammaclass}
The following holds.
\begin{enumerate}
\item The restriction $g \cap S \times \R \times X$ is analytic.
\item The set $g$ is the graph of a map with range $\mathcal{T}_e$.
\item $t \mapsto g(y,t)$ is a $d_L$ $1$-Lipschitz $G$-order preserving for $y \in \mathcal{T}$.
\item $(t,y) \mapsto g(y,t)$ is bijective on $\mathcal{T}$, and its inverse is
\[
x \mapsto g^{-1}(x) = \big( f(y),\pm d_L(x,f(y)) \big)
\]
where $f$ is the quotient map of Proposition \ref{P:sicogrF} and the positive/negative sign depends on $x \in G(f(y))$/$x \in G^{-1}(f(y))$.
\end{enumerate}
\end{proposition}
Another property of $d_{L}$-cyclically monotone transference plans.
\begin{proposition}
\label{P:ortho}
For any $\pi$ $d_{L}$-monotone there exists a $d_L$-cyclically monotone transference plan $\tilde \pi$ with the same cost of $\pi$ such that it coincides with the identity on $\mu \wedge \nu$.
\end{proposition}

Coming back to the abstract Wiener space, we have that given $\mu,\nu \ll\gamma$ and given $\pi \in \Pi(\mu,\nu)$ $\norm_{H(\gamma)}$-cyclically monotone, 
we have constructed the transport $\mathcal{T}$ (and $\mathcal{T}_{e}$), an equivalence relation $R$ on it with geodesics as equivalence classes  
and the corresponding disintegration
is strongly consistent:
\begin{equation}\label{E:primadis}
\mu\llcorner_{\mathcal{T}} =  \int_{\mathcal{S}} \mu_{y} m(dy)
\end{equation}
with $m = f_{\sharp} \mu$ and $\mu_{y}(R(y)) = 1$ for $m$-a.e. $y \in \mathcal{T}$. Using the ray map $g$ one can assume that $\mu_{y} \in \mathcal{P}(\erre)$ and 
\[
\mu\llcorner_{\mathcal{T}} = g_{\sharp}   \int_{\mathcal{S}} \mu_{y} m(dy) .
\]

\section{Regularity of disintegration}
\label{S:disin}
To obtain existence of an optimal transport map it is enough to prove that $\mu$ in concentrated on $\mathcal{T}$ and 
$\mu_{y}$ is a continuous measure for $m$-a.e. $y \in \mathcal{S}$. Indeed at that point, for every $y \in \mathcal{S}$ we consider 
the unique monotone map $T_{y}$ such that $T_{y\,\sharp}\mu_{y}=\nu_{y}$, then $T(g(y,t)): = T_{y}(g(y,t))$ is an optimal transport map,
see Theorem 6.2 of \cite{biacava:streconv}.

In Section \ref{S:disin} we introduce the fundamental regularity assumption (Assumption \ref{A:NDEatom}) on the measure $\gamma$ 
and we show that it implies the $\gamma$-negligibility of the set of initial points. Consequently we obtain a disintegration of $\mu$
on the whole space. 
From Assumption \ref{A:NDEatom} it also follows  that the disintegration of $\mu$ w.r.t. the ray equivalence relation $R$ 
has continuous conditions probabilities.

Define the map $ X\times X \ni (x,y) \mapsto  T_{t}(x,y) : =  x (1-t ) + y t$.

\begin{assumption}[Non-degeneracy assumption]
\label{A:NDEatom} The measure $\gamma$ is said to satisfy Assumption \ref{A:NDEatom} w.r.t. a $\norm_{H(\gamma)}$-cyclically monotone set 
$\Gamma$ if
\begin{itemize}
\item[i)] $\pi(\Gamma)=1$ with $\pi \in \Pi(\mu,\nu)$ and $\mu,\nu \ll \gamma$;
\item [ii)] for each closed set $A$ such that $\mu (A) > 0$ there exists $C >0 $ and $\{ t_{n} \}_{n\in \enne} \subset [0,1]$ converging to $0$ as 
$n\to +\infty$ such that 
\[
\gamma(  T_{t_{n}}(\Gamma \cap A\times X) ) \geq C \mu (A)
\]
for all $n \in \enne$. 
\end{itemize}
\end{assumption}

Clearly it is enough to verify Assumption \ref{A:NDEatom} for $A$ compact set. 

An immediate consequence of the Assumption \ref{A:NDEatom} is that the final points are $\gamma$-negligible. 
\begin{proposition}
\label{P:puntini}
If $\gamma$ satisfies Assumption \eqref{A:NDEatom} then
\[
\mu( a(\mathcal{T}_e)) = 0.
\]
\end{proposition}

\begin{proof}
Let $A= a(\mathcal{T}_{e})$. Suppose by contradiction $\mu( A)>0$.
By inner regularity and $\Gamma \subset \{ (x,y) : \|x-y\|_{H(\gamma)}< + \infty \}$, 
there exists a compact set $\hat A \subset A$ such that $\mu (\hat A)>0$ and  $\rho_{1}(x)\geq \delta$ for all $x\in \hat A$ and for some constant $\delta>0$.
Moreover we can assume that
\[
\Gamma \cap \hat A \times X \subset \{ (x,y) : \|x-y\|_{H(\gamma)}\leq M \}
\] 
for some positive $M \in \erre$. 

By Assumption \ref{A:NDEatom} there exist $C > 0$ and 
$\{ t_{n} \}_{n\in \enne}$ converging to 0 such that
\[
\gamma(  T_{t_{n}}(\Gamma \cap \hat A \times X) ) \geq C \mu ( \hat A ) \geq \delta C  \gamma(\hat A).
\]
Denote with $\hat A_{t_{n}} = T_{t_{n}}(\Gamma \cap \hat A \times X)$ and define $ \hat A^{\ve}:= \big\{ x : \| \hat A -x\|_{H(\gamma)} < \ve \big\}$. 
Since $\hat A \subset A = a(\mathcal{T}_{e})$,  $\hat A_{t_{n}} \cap \hat A = \emptyset$ for every $n \in \enne$. 
Moreover for $t_{n}\leq\ve/M$ it holds $\hat A^{\ve} \supset \hat A_{t_{n}}$. So we have for $t_{n}$ small enough
\[ 
\gamma(\hat A^{\ve}) \geq \gamma(\hat A) + \gamma(\hat A_{t_{n}})  \geq (1+ C\delta) \gamma (\hat A).
\]
Since $\gamma(\hat A) = \lim_{\ve \to 0} \gamma (\hat A^{\ve})$, this is a contradiction.
\end{proof}

It follows that $\mu(\mathcal{T}) = 1$, therefore we can use the Disintegration Theorem \ref{T:disintr} to write
\begin{equation}
\label{E:disintT}
\mu = \int_S \mu_y m(dy), \quad m = f_\sharp \mu, \ \mu_y \in \mathcal{P}(R(y)).
\end{equation}
The disintegration is strongly consistent since the quotient map $f : \mathcal T \to \mathcal T$ is $\mu$-measurable and $(\mathcal T,\mathcal B(\mathcal T))$ is countably generated. 

The second consequence of Assumption \ref{A:NDEatom} is that $\mu_y$ is continuous, i.e. $\mu_y(\{x\}) = 0$ for all $x \in X$.

\begin{proposition}
\label{P:nonatoms}
If $\gamma$ satisfies Assumption \ref{A:NDEatom} then the conditional probabilities $\mu_y$ are continuous for $m_{\gamma}$-a.e. $y \in S$.
\end{proposition}

\begin{proof}
From the regularity of the disintegration and the fact that $m(S) = 1$, we can assume that the map $y \mapsto \mu_y$ is weakly continuous on a 
compact set $K \subset S$ of comeasure $<\ve$. 
It is enough to prove the proposition on $K$.

{\it Step 1.} From the continuity of $K \ni y \mapsto \mu_y \in \mathcal{P}(X)$ w.r.t. the weak topology, it follows that the map
\[
y \mapsto A(y) := \big\{ x \in R(y): \mu_y(\{x\}) > 0 \big\} = \cup_n \big\{ x \in R(y): \mu_y(\{x\}) \geq 2^{-n} \big\}
\]
is $\sigma$-closed: in fact, if $(y_m,x_m) \to (y,x)$ and $\mu_{y_m}(\{x_m\}) \geq 2^{-n}$, then $\mu_y(\{x\}) \geq 2^{-n}$ by u.s.c. on compact sets.
Hence $A$ is Borel.

{\it Step 2.} The claim is equivalent to $\mu( P_{2}(A))=0$. Suppose by contradiction $\mu(P_{2}(A))>0$.
By Lusin Theorem (Theorem 5.8.11 of \cite{Sri:courseborel}) 
$A$ is the countable union of Borel graphs.
Therefore we can take a Borel selection of $A$ just considering one of the Borel graphs, say $\hat A$. 
Clearly $m (P_{1}(\hat A)) > 0$ hence by \eqref{E:disintT} $\mu(P_{2}(\hat A))>0$. 
By Assumption \ref{A:NDEatom}  $\gamma( T_{t_{n}}(\Gamma \cap  P_{2}(\hat A) \times X) ) \geq C \mu(P_{2}(\hat A))$ for some $C>0$ and $t_{n}\to 0$.
From $T_{t_{n}}(\Gamma \cap P_{2}(\hat A) \times X) \cap (P_{2}(\hat A)) = \emptyset$,  using the same argument of Proposition \ref{P:puntini}, the claim follows.
\end{proof}

\section{An approximation result}
\label{S:approx}

In Section \ref{S:disin} we proved that if $\gamma$ satisfies Assumption \ref{A:NDEatom} w.r.t. a $\norm_{H(\gamma)}$-cyclically monotone set 
$\Gamma$ such that $\pi(\Gamma)=1$, then the disintegration of $\mu$ has enough regularity to solve the Monge minimization problem. 

The remaining part of the paper will be devoted to proving that the centred non-degenerate Gaussian measure $\gamma$ 
verifies Assumption \ref{A:NDEatom} w.r.t. a $\norm_{H(\gamma)}$-cyclically monotone set 
$\Gamma$ such that $\pi(\Gamma)=1$ with $\pi \in \Pi(\mu,\nu)$ and $\mu,\nu \ll \gamma$.
In particular in this section we prove that it is enough to verify Assumption \ref{A:NDEatom} for a well prepared finite dimensional approximation, 
provided some uniformity holds.

Let $P_{d} : X \to H$ be the projection map of Proposition \ref{P:appro} associated to the orthonormal basis $\{e_{i}\}_{i \in \enne}$ of $H(\gamma)$  
with $e_{i} = R_{\gamma} \hat e_{i}$ for $\hat e_{i} \in R_{\gamma}^{*}X^{*}$ and  $P_{d\,\sharp} \gamma = \gamma_{d}$.

Consider the following measures 
\begin{equation}\label{E:approxvera}
\mu_{d} : = P_{d\,\sharp} \mu, \qquad \nu_{d} : = P_{d\,\sharp} \nu
\end{equation}
and observe that $\mu_{d}= \rho_{1,d} \gamma_{d}$ and $\nu_{d}= \rho_{2,d} \gamma_{d}$
with 
\begin{equation}\label{E:densita}
\rho_{i,d}(z) = \int \rho_{i}(x) \gamma_{z,d}^{\perp} (dx), \qquad i =1,2,
\end{equation}
where $\gamma_{z,d}$ is defined in \ref{E:disintg}. Recall that $\mu_{d} \weak \mu$ and $\nu_{d} \weak \nu$ as $d \nearrow + \infty$.
Since $\rho_{i,d}$ depend only on the first $d$-coordinates, the measures $\mu_{d},\nu_{d}$
can be considered as probability measure on $\erre^{d}$. Therefore
we can study the transport problem with euclidean norm cost $\|x \|_{d}^{2} : = \sum_{j=1}^{d} x_{j}^{2}$:
\begin{equation}\label{E:traspapprox}
\min_{\pi \in \Pi(\mu_{d},\nu_{d})} \int \|x-y \|_{d} \pi(dx dy).
\end{equation}
It is a well-known fact in optimal transportation that this problem has a minimizer of the form $(Id,T_{d})_{\sharp}\mu_{d}$ 
with $T_{d}$ invertible and Borel. 
For each $d$ we choose as optimal map $T_{d}$ the one obtained gluing the monotone rearrangements over the geodesics.

%

%


\begin{proposition}\label{P:limiteottimo}
Let $\pi_{d} \in \Pi(\mu_{d},\nu_{d})$ be an optimal transference plan for \eqref{E:traspapprox} 
such that $\pi_{d} \weak \pi$. 
Then $\pi \in \Pi(\mu,\nu)$ is an optimal transport plan for \eqref{E:trasp}.
\end{proposition}
\begin{proof}
Let $\hat \pi \in \Pi(\mu,\nu)$ be an optimal transference plan. The following holds true
\begin{align*}
\int \|x-y \|_{H} \hat \pi(dxdy) \geq &~ \int \|P_{d}(x-y) \|_{H(\gamma)} \hat \pi(dxdy) = \int \|x-y \|_{H(\gamma)} ((P_{d}\otimes P_{d})_{\sharp}\hat \pi)(dxdy)  \crcr
 \geq &~ \int \|x-y \|_{H(\gamma)} \pi_{d}(dxdy).
\end{align*}
Since $\norm_{H(\gamma)}$ is l.s.c. and $\pi_{d} \weak \pi$ it follows that 
\[
\int \|x-y \|_{H(\gamma)} \hat \pi(dxdy) \geq \liminf_{d \to + \infty} \int \|x-y \|_{H(\gamma)} \pi_{d}(dxdy) \geq \int \|x-y \|_{H(\gamma)}\pi(dxdy).
\]
Hence the claim follows.
\end{proof}

If the sequence $\pi_{d} \in \Pi(\mu_{d},\nu_{d})$ of optimal transference plans
satisfies, for every $d \in \enne$, Assumption \ref{A:NDEatom} w.r.t. $\Gamma_{d} = \gr(T_{d})$  with $C$ independent on $d$,
then the optimal transference plan $\pi$, weak limit of $\pi_{d}$, satisfies Assumption \ref{A:NDEatom} w.r.t. a $\norm_{H(\gamma)}$-cyclically monotone set $\Gamma$ such that $\pi(\Gamma)=1$.


\begin{theorem}\label{T:approx}
Assume that there exists $C > 0$ such that for all $d \in \enne$ and for all $A \subset X$ compact set 
the following holds true
\[
\gamma_{d} \big(  T_{t}( \Gamma_{d} \cap A \times X) \big) \geq C \mu_{d}( A).
\]
Then for all $A \subset X$ compact set
\begin{equation}\label{E:stima}
\gamma \big( T_{t}( \Gamma \cap A \times X)\big) \geq C \mu( A ),
\end{equation}
where $\Gamma \subset X \times X$ is $\norm_{H(\gamma)}$-cyclically monotone with $\pi(\Gamma)=1$.
\end{theorem}

\begin{proof}
Assume that $A = C(B) \in \mathcal{E}(X)$ with $B \in \erre^{n}$ compact set for some fixed $n \in \enne$.
It follows from Proposition \ref{P:limiteottimo} 
that $\pi$ is an optimal transference plan concentrated on a $\norm_{H(\gamma)}$-cyclically monotone set $\Gamma$.
Observe that $\Gamma_{d}$ is closed for every $d \in \enne$ .

{\it Step 1.}
Since $\mu_{d} \weak \mu$ and $\nu_{d} \weak \nu$, for every $\ve >0$ there exist $K_{1,\ve}$ and $K_{2,\ve}$ compact sets such that 
$\mu_{d}(K_{1,\ve}) \geq 1- \ve/2$ and  $\nu_{d}(K_{2,\ve}) \geq 1- \ve/2$. 
Consider the compact se $\Gamma_{d,\ve} := \Gamma_{d} \cap K_{1,\ve} \times K_{2,\ve}$, 
then $\pi_{d}(\Gamma_{d,\ve}) \geq 1-\ve$ and $\Gamma_{d,\ve}$ converges in the Hausdorff topology, 
up to subsequences for $d \to + \infty$, to a set $\Gamma_{\ve} \subset \Gamma$ with $\pi(\Gamma_{\ve})\geq 1-\ve$. 

Since $T_{t}$ is continuous, $T_{t}(\Gamma_{d,\ve} \cap A \times X)$ is compact and by 
\[
T_{t}(\Gamma_{d,\ve} \cap A\times X) \subset \overline{co} (P_{1}(K_{\ve} \cap A\times X) \cap P_{2}(K_{\ve} \cap A\times X))
\]
it follows that $T_{t}(\Gamma_{d,\ve} \cap A\times X)$ converges in the Hausdorff topology to $T_{t}(\Gamma_{\ve} \cap A\times X)$.

{\it Step 2.}
It follows that
\[
\gamma(T_{t}(\Gamma_{\ve} \cap A \times X)) \geq  \limsup_{d \to + \infty}  \gamma_{d}(T_{t}(\Gamma_{d,\ve} \cap A \times X)),
\]
hence, using the fact that $\Gamma_{d,\ve}$ is a subset of a graph, it follows that
\begin{align}\label{E:bene}
 \gamma(T_{t}(\Gamma_{\ve} \cap A \times X)) \geq &~ \limsup_{d \to + \infty} \gamma_{d}(T_{t}(\Gamma_{d,\ve} \cap A \times X))\crcr
\geq &~ C \limsup_{d \to + \infty} \mu_{d} ( P_{1}(\Gamma_{d,\ve}) \cap A) \crcr
\geq &~ C \limsup_{d \to + \infty} \mu_{d} (A) - C \ve
\end{align}
where in the last equation we have used $\mu_{d} (P_{1}(\Gamma_{d,\ve})) \geq 1-\ve$.

%

{\it Step 3.} Since $\mu_{d}= P_{d\,\sharp} \mu$
\begin{equation}\label{E:bbene}
\gamma(T_{t}(\Gamma_{\ve} \cap A \times X)) \geq C \limsup_{d \to + \infty} \mu_{d} (A) - C \ve = 
C\mu(A) - C\ve.
\end{equation}
for all $A \in \mathcal{E}(X)$ with finite dimensional base; note that this family of sets is a base for the weak topology. 
It follows from the compactness of $\Gamma_{\ve}$ that \eqref{E:bbene} is stable under uncountable intersection, 
hence \eqref{E:bbene} holds true for all weak closed subset of $\Gamma_{\ve}$. Since $\Gamma_{\ve}$ is a compact set, 
weak topology has the same closed set of the strong one. It follows that for all closed set $A \subset X$ 
\[
\gamma(T_{t}(\Gamma \cap A \times X)) \geq \gamma(T_{t}(\Gamma_{\ve} \cap A \times X)) \geq C \mu ( A) - C \ve.
\]
Passing to the limit as $\ve \to 0$, the claim follows.
\end{proof}

\section{Finite dimensional estimate}
\label{S:finite}
Throughout this Section we will use the following notation $Jac(F)(x) : = |\det dF|(x)$ for any map $F:\erre^{d} \to \erre^{d}$.

In the next theorem we prove that $d$-dimensional standard Gaussian measure $\gamma_{d} =  P_{d\,\sharp} \gamma$ satisfies 
Assumption \ref{A:NDEatom} for $\Gamma =  \gr(T_{d})=\Gamma_{d}$:
\[
\gamma_{d} \big(  T_{t}( \Gamma_{d} \cap A \times X) \big) \geq C \mu_{d}( A ).
\]
Observe that the set $T_{t}( \Gamma_{d} \cap A \times X)$ is parametrized by the map $T_{d,t}:= Id(1-t) + T_{d} t$.
  
\begin{theorem}\label{T:stimapprox}
Assume that there exists $M>0$ such that 
 $\rho_{i,d} (x) \leq C$ for $\gamma_{d}$-a.e. $x \in \erre^{d}$ and $i=1,2$. 
Then the following estimate holds true 
\[
\gamma_{d} \big(  T_{d,t}( A ) \big) \geq \frac{1}{C} \mu_{d}(A ).
\]
\end{theorem}
 
\begin{proof} During the proof we will omit the subscript $d$.

{\it Step 1.} Consider  the Monge minimization problem with cost $c_{p}$, \eqref{E:monge},
between $\mu_{d}$ and $\nu_{d}$. It follows from Theorem \ref{T:different} and the boundedness of $\rho_{i,d}$ that 
there exists a unique  optimal map $T_{p}$ approximately differentiable $\mu_{d}$-a.e., hence
by Lemma \ref{L:density} it follows that
\[
\rho_{2}(T_{p}(x)) |\det \tilde \nabla T_{p}|(x)\prod_{j=1}^{d}\frac{1}{\sqrt{2\pi}} \exp\Big\{ -  \frac{T_{p}(x)_{j}^{2}}{2} \Big\} = 
\rho_{1}(x) \prod_{j=1}^{d}\frac{1}{\sqrt{2\pi}} \exp\Big\{ -  \frac{x_{j}^{2}}{2} \Big\}.
\]
Since for $\mu_{d}$-a.e. $x \in\erre^{d}$ $|\det \tilde \nabla T_{p}|(x)>0$, also $\rho_{2}(T_{p}(x))>0$ for $\mu_{d}$-a.e. $x \in \erre^{d}$. 
Hence the following makes sense $\mu$-a.e.:
\[
Jac(T_{p})(x)= |\det \tilde \nabla T_{p}|(x) = \frac{\rho_{1}(x)}{\rho_{2}(T_{p}(x))} \exp \Big\{ \sum_{j=1}^{d} - \frac{1}{2} ( x_{j}^{2} - T_{p}(x)_{j}^{2} )  \Big\}.
\]

{\it Step 2.} Let $T_{p,t} : = Id (1-t) + T_{p} t$.
From Theorem \ref{T:different}, $\det \tilde \nabla T_{p}(x) = \prod_{j=1}^{d}\lambda_{j}$ with $\lambda_{i}>0$ for $i=1,\dots,d$.
It follows that
\[
Jac(T_{p,t})(x) = \det (Id (1-t) + \tilde\nabla T_{p}(x)t) = \prod_{j=1}^{d}\big( (1-t ) + \lambda_{j}t \big).
\]
Passing to logarithms, we have by concavity 
\[
\log (Jac(T_{p,t})(x) ) \geq t \log (Jac(T_{p})(x) ) \Longrightarrow Jac(T_{p,t})(x) \geq  (Jac(T_{p})(x))^{t}.
\]
Hence
\begin{equation}
Jac(T_{p,t})(x) \geq \Big(\frac{\rho_{1}(x)}{\rho_{2}(T_{p}(x))}\Big)^{t} \exp \Big\{ \sum_{j=1}^{d} - \frac{1}{2} t ( x_{j}^{2} - T_{p}(x)_{j}^{2} )  \Big\}.
\end{equation}

{\it Step 3.} 
We have the following
\begin{align*}
 \exp\Big\{  \sum_{j=1}^{d}- &~\frac{1}{2} (T_{p,t}(x)_{j}^{2}  - x_{j}^{2}) \Big\} Jac (T_{p,t})(x)  \crcr
&~ \geq  \exp\Big\{  \sum_{j=1}^{d}- \frac{1}{2} (T_{p,t}(x)_{j}^{2}  - x_{j}^{2}) \Big\} 
 \Big(\frac{\rho_{1}(x)}{\rho_{2}(T_{p}(x))}\Big)^{t} \exp \Big\{ \sum_{j=1}^{d} - \frac{1}{2} t ( x_{j}^{2} - T_{p}(x)_{j}^{2} )  \Big\} \crcr
&~ =  \Big(\frac{\rho_{1}(x)}{\rho_{2}(T_{p}(x))}\Big)^{t}  
\exp\Big\{  \sum_{j=1}^{d}- \frac{1}{2} (T_{p,t}(x)_{j}^{2}  - x_{j}^{2}   +t x_{j}^{2} - t T_{p}(x)_{j}^{2}    ) \Big\} \crcr
&~ =  \Big(\frac{\rho_{1}(x)}{\rho_{2}(T_{p}(x))}\Big)^{t}  
\exp\Big\{  \sum_{j=1}^{d}- \frac{1}{2} \Big( ((1-t)x_{j} + t T_{p}(x)_{j})^{2}  - ( (1-t) x_{j}^{2}   + t T_{p}(x)_{j}^{2}    \Big) \Big\} \crcr
&~ =  \Big(\frac{\rho_{1}(x)}{\rho_{2}(T_{p}(x))}\Big)^{t}  
\exp\Big\{  \sum_{j=1}^{d}- \frac{1}{2} \big( x_{j} -T_{p}(x)_{j} \big)^{2} (t^{2}-t) \Big\}  \crcr
&~ =  \Big(\frac{\rho_{1}(x)}{\rho_{2}(T_{p}(x))}\Big)^{t}  
\exp\Big\{  - \frac{1}{2} \| x -T_{p}(x) \|_{d}^{2} (t^{2}-t) \Big\}. 
\end{align*}

Hence
\begin{align*}
\gamma(T_{p,t}(A)) = &~ \int_{A}  Jac(T_{p,t})(x)  \prod_{j=1}^{d}\frac{1}{\sqrt{2\pi}} \exp\Big\{ - \frac{1}{2} T_{p,t}(x)_{j}^{2} \Big\} \mathcal{L}^{d}(dx) \crcr
= &~ \int_{A}  Jac(T_{p,t})(x)  \exp\Big\{ \sum_{j=1}^{d} - \frac{1}{2} ( T_{p,t}(x)_{j}^{2} - x_{j}^{2} ) \Big\} \gamma(dx) \crcr
\geq &~ \int_{A} \Big(\frac{\rho_{1}(x)}{\rho_{2}(T_{p}(x))}\Big)^{t}  \exp\Big\{  \frac{1}{2} \| x -T_{p}(x) \|_{d}^{2} (t- t^{2}) \Big\} \gamma(dx) \crcr
\geq &~ \frac{1}{C^{t}}\int_{A} \rho_{1}(x)^{t} \gamma(dx)  \crcr
\geq &~ \frac{1}{C^{t}}\int_{A}\rho_{1}(x)^{t-1} \mu(dx)  \crcr
\geq &~ \frac{1}{C}\mu(A).
\end{align*}
The claim follows.

{\it Step 4.} Since $(Id,T_{p})_{\sharp}\mu_{d} \weak (Id,T)_{\sharp} \mu_{d}$ as $p \searrow 1$, using the same techniques of Theorem \ref{T:approx}'s proof,
it is fairly easy to prove that 
\[
\gamma_{d}(T_{t}(A)) \geq \frac{1}{C} \mu_{d}( A).
\]
\end{proof}

\begin{remark}\label{R:resume}
We summarize the results obtained so far. If $\rho_{1},\rho_{2} \leq M$, 
then from \eqref{E:densita} it follows that the densities of $\mu_{d}$ and $\nu_{d}$ enjoy the same property with the same constant $M$.				
Hence using the approximating sequences $\mu_{d}$ and 
$\nu_{d}$, we have from Theorem \ref{T:approx} and Theorem \ref{T:stimapprox} that  
\[
\gamma \big( T_{t}( \Gamma \cap A \times X)\big) \geq \frac{1}{M} \mu( A). 
\]
As Proposition \ref{P:puntini} and Proposition \ref{P:nonatoms} show, this estimate implies $\mu(a(\mathcal{T}))=0$ and the continuity of the conditional 
probabilities $\mu_{y}$.
Since the optimal finite dimensional map $T_{d}$ is invertible, following the argument of Theorem \ref{T:stimapprox} we can also prove
\begin{equation}\label{E:ricapitolo}
\gamma \big( T_{1-t}( \Gamma \cap X \times A)\big) \geq \frac{1}{M} \nu( A),
\end{equation}
and adapting the proofs of Proposition \ref{P:puntini} and Proposition \ref{P:nonatoms} we can prove that  $\nu(b(\mathcal{T}))=0$ and 
the continuity of the conditional probabilities $\nu_{y}$. So we have
\[
\mu = \int_{\mathcal{S}} \mu_{y} m (dy), \quad \nu = \int_{\mathcal{S}} \nu_{y} m (dy), \qquad \mu_{y}, \nu_{y} \ \textrm{continuous for } m-a.e. y \in \mathcal{S}. 
\]


In the next Section we remove the hypothesis $\rho_{1},\rho_{2} \leq M$.
\end{remark}

\section{Solution}
\label{S:solu}

In this Section we obtain the existence of an optimal transport map for the Monge minimization problem between measures $\mu$ and $\nu$ 
absolute continuous w.r.t. $\gamma$. We first prove that the set of initial points and the set of final points have $\mu$-measure zero and $\nu$-measure zero respectively.
Recall that the initial points map $a$ and the final points map $b$ have been introduced in Definition \ref{D:endpoint}. 

\begin{proposition}\label{P:noiniz}
Let $\mu, \nu \in \mathcal{P}(X)$ be such that $\mu,\nu \ll \gamma$. Then $\mu(a(\mathcal{T}))= \nu(b(\mathcal{T}))=0$. 
\end{proposition}
\begin{proof} Let $\mu = \rho_{1}\gamma$ and $\nu =\rho_{2}\gamma$.
We prove that $\mu(a(\mathcal{T}))= 0$.

{\it Step 1.} Assume by contradiction that $\mu(a(\mathcal{T}))>0$. 
Let $A \subset a(\mathcal{T})$ be such that $\mu(A)>0$ and for every $x \in A$,
$\rho_{1}(x)\leq M$ for some positive constant $M$.
Consider $\gamma\llcorner_{\mathcal{T}}$ and its disintegration 
\[
\gamma\llcorner_{\mathcal{T}} = \int_{\mathcal{S}} \gamma_{y} m_{\gamma}(dy), \quad \gamma_{y}(\mathcal{T})=1, \ m_{\gamma}-a.e. y \in \mathcal{S}.
\]
Consider the initial point map $a : \mathcal{S} \to A$ and the measure $a_{\sharp} m_{\gamma}$. Observe that since 
\[
\forall B \subset A: \mu(B)>0 \quad \Rightarrow \quad \gamma( R(B) \cap \mathcal{T})>0,
\]
it follows that $\mu\llcorner_{A} \ll a_{\sharp} m_{\gamma}$. Hence there exists $\hat A \subset A$ of positive $a_{\sharp}m_{\gamma}$-measure such that 
the map 
\[
\hat A\ni x \mapsto h(x) : = \frac{d \mu\llcorner_{A}}{d a_{\sharp}m_{\gamma}} (x)
\]
verifies $ h(x) \leq M'$ for some positive constan $M'$. 

{\it Step 2.}
Considering 
\[
\mu\llcorner_{\hat A}, \qquad \hat \gamma : = \int_{R(\hat A) \cap \mathcal{S}} h(a(y))\gamma_{y} m_{\gamma}(dy), 
\]
we have the claim. Indeed both have uniformly bounded densities w.r.t. $\gamma$ and $\mathcal{T}_{e}$ is still a transport set 
for the transport problem between $\mu\llcorner_{\hat A}$ and $\hat \gamma$.
Indeed for $S \subset \mathcal{S}$
\begin{align*}
\mu\llcorner_{\hat A}(\cup_{y\in S} R(y)) = &~ \mu\llcorner_{\hat A} (a(S)) \crcr
= &~ \int_{a(S)} h(a) (a_{\sharp}m_{\gamma})(da) \crcr
= &~ \int_{S} h(a(y)) m_{\gamma}(dy) = \hat \gamma (\cup_{y\in S} R(y)).
\end{align*}

Hence we can project the measures, obtain the finite dimensional 
estimate of Theorem \ref{T:stimapprox}, obtain the infinite dimensional estimate through Theorem \ref{T:approx} and finally by Proposition 
\ref{P:puntini} get that $\mu(\hat A)=0$, that is a contradiction with $\mu(\hat A)>0$. In the same way, following Remark \ref{R:resume}, 
we obtain that $\nu(b(\mathcal{T}))=0$.
\end{proof}

It follows that the disintegration formula \eqref{E:primadis}
holds true on the whole transportation set:
\[
\mu = \int \mu_{y} m(dy), \qquad \nu = \int \nu_{y} m(dy).
\]

\begin{proposition}\label{P:noatom}
For $m$-a.e. $y \in \mathcal{S}$ the conditional probabilities $\mu_{y}$ and $\nu_{y}$ have no atoms.
\end{proposition}
\begin{proof} We only prove the claim for $\mu_{y}$.

{\it Step 1.} Suppose by contradiction that there exist a measurable set $\hat{\mathcal{S}} \subset \mathcal{S}$ 
such that $m(\hat{ \mathcal{S}})>0$ 
and for every $y \in \hat{ \mathcal{S}}$ there exists $x(y)$ such that $\mu_{y}(\{x(y)\})>0$. 
Restrict and normalize both $\mu$ and $\nu$ to $R(\hat{\mathcal{S}})$, and denote them again with $\mu$ and $\nu$.

Consider the sets $K_{i,M} : = \{ x \in X :  \rho_{i}\leq M \}$ for $i = 1,2$. 
Note that 
$\mu(K_{1,M}) \geq 1 -c_{1}(M)$ and $\nu(K_{2,\delta}) \geq 1 -c_{2}(M)$ with $c_{i}(M) \to 0$ as $M \nearrow +\infty$.
Hence for $M$ sufficiently large the conditional probabilities of the disintegration of $\mu\llcorner_{K_{1,M}}$ 
have atoms, therefore we can assume, possibly restricting $\hat{\mathcal{ S}}$, that for all $y \in \hat{\mathcal{S}}$ it holds $x(y) \in K_{1,M}$.

{\it Step 2.}
Define
\[
\mu_{y,M} : = \mu_{y}\llcorner_{K_{1,M}}, \qquad  \nu_{y,M} : = \nu_{y}\llcorner_{K_{2,M}},
\]
and introduce the set
\[
D(N) : = \Big\{ y \in \hat{\mathcal{S}} : \frac{\mu_{y,M}(R(y))} {\nu_{y,M}(R(y))} \leq N \Big\}.
\]
Then for sufficiently large $N$, $m(D(N)) > 0$. The map
$ D(N)\ni y  \mapsto h(y) : =   \nu_{y,M}(R(y))/ \mu_{y,M}(R(y)) \leq N$ permits to define 
\[
\hat \mu: = \int_{D(N)} h(y) \mu_{y,M} m (dy), \qquad \hat \nu : = \nu\llcorner_{R(D(N))\cap K_{2,M} }.
\]
It follows that $\hat \mu$ and $\hat \nu$ have bounded densities w.r.t. $\gamma$ and the set 
$\hat{\mathcal{T}}: = \mathcal{T} \cap G(K_{1,\delta}) \cap G^{-1}(K_{2,\delta})$ is 
a transport set for the transport problem between $\hat \mu$ and $\hat \nu$. 

It follows from Theorem \ref{T:approx} and Theorem \ref{T:stimapprox} that $\hat \gamma:=\gamma\llcorner_{\hat {\mathcal{T}}}$ 
verifies Assumption \ref{A:NDEatom} w.r.t. $G \cap K_{1,M}\times X \cap X \times K_{2,M}$. 
Therefore from Proposition \ref{P:nonatoms} follows that the conditional probabilities
$\hat \mu_{y}$ of the disintegration of $\hat \mu$ are continuous. Since $\hat \mu_{y} =c(y) \mu_{y}\llcorner_{\hat{\mathcal{T}}}$
for some positive constant $c(y)$, we have a contradiction.
\end{proof}


It follows straightforwardly the existence of an optimal invertible transport map. 

\begin{theorem}\label{T:esiste}
Let $\mu,\nu \in \mathcal{P}(X)$ absolute continuous w.r.t. $\gamma$ and assume that there exists $\pi \in \Pi(\mu,\nu)$ such that $\mathcal{I}(\pi)$ is finite. 
Then there exists a solution for the Monge minimization problem
\[
\min_{T : T_{\sharp}\mu=\nu} \int_{X} \| x-T(x) \|_{H(\gamma)} \mu(dx).
\]
Moreover we can find $T$ invertible.
\end{theorem}
\begin{proof}
For $m$-a.e $y \in \mathcal{S}$ $\mu_{y}$ and $\nu_{y}$ are continuous. Since $R(y)$ is one dimensional and 
the ray map $\ni \erre t \mapsto g(t,y)$ is an isometry w.r.t. $\norm_{H(\gamma)}$, we can define 
the non atomic measures $g(y,\cdot)_{\sharp}\mu_{y},g(y,\cdot)_{\sharp}\nu_{y} \in \mathcal{P}(\erre)$. By the one-dimensional theory, there exists 
a monotone map $T_{y}:\erre \to \erre$ such that 
\[
T_{y\,\sharp} \Big(g(y,\cdot)_{\sharp}\mu_{y} \Big) = g(y,\cdot)_{\sharp}\nu_{y}.
\]
Using the inverse of the ray map, we can define $T_{y}$ on $R(y)$. Hence for $m$-a.e. $y \in \mathcal{S}$ we have a $\norm_{H(\gamma)}$-cyclically monotone map $T_{y}$ such that $T_{y\,\sharp}\mu_{y}=\nu_{y}$.
To conclude define $T : \mathcal{T} \to \mathcal{T}$ such that $T = T_{y}$ on $R(y)$. Indeed $T$ is $\mu$-measurable, invertible and $T_{\sharp}\mu=\nu$.
For the details, see the proof of Theorem 6.2 of \cite{biacava:streconv}.
\end{proof}

\section{Notation}
\label{S:notation}

\begin{tabbing}
\hspace{4cm}\=\kill
$P_{i_1\dots i_I}$ \> projection of $x \in \Pi_{k=1,\dots,K} X_k$ into its $(i_1,\dots,i_I)$ coordinates, keeping order
\\
$\mathcal{P}(X)$ or $\mathcal{P}(X,\Omega)$ \> probability measures on a measurable space $(X,\Omega)$
\\
$\mathcal{M}(X)$ or $\mathcal{M}(X,\Omega)$ \> signed measures on a measurable space $(X,\Omega)$
\\
$f \llcorner_A$ \> the restriction of the function $f$ to $A$
\\
$\mu \llcorner_A$ \> the restriction of the measure $\mu$ to the $\sigma$-algebra $A \cap \Sigma$
\\
$\mathcal{L}^d$ \> Lebesgue measure on $\R^d$
\\
$\mathcal{H}^k$ \> $k$-dimensional Hausdorff measure
\\
$\Pi(\mu_1,\dots,\mu_I)$ \> $\pi \in \mathcal{P}(\Pi_{i=1}^I X_i, \otimes_{i=1}^I \Sigma_i)$ with marginals $(P_i)_\sharp \pi = \mu_i \in \mathcal{P}(X_i)$
\\
$\mathcal{I}(\pi)$ \> cost functional \eqref{E:Ifunct}
\\
$c$ \> cost function $ : X \times Y \mapsto [0,+\infty]$
\\
$\mathcal{I}$ \> transportation cost \eqref{E:Ifunct}
\\
$\phi^c$ \> $c$-transform of a function $\phi$ \eqref{E:ctransf}
\\
$\partial^c \f$ \> $d$-subdifferential of $\f$ \eqref{E:csudiff}
\\
$\Phi_c$ \> subset of $L^1(\mu) \times L^1(\nu)$ defined in \eqref{E:Phicset}
\\
$J(\phi,\psi)$ \> functional defined in \eqref{E:Jfunct}
\\
$C_b$ or $C_b(X,\R)$ \> continuous bounded functions on a topological space $X$
\\
$(X,d)$ \> Polish space
\\
$(X,d_L)$ \> non-branching geodesic separable metric space
\\
$D_L(x)$ \> the set $\{y : d_L(x,y) < +\infty\}$
\\
$\gamma_{[x,y]}(t)$ \> geodesics $\gamma : [0,1] \to X$ such that $\gamma(0) = x$, $\gamma(1) = y$
\\
$B_r(x)$ \> open ball of center $x$ and radius $r$ in $(X,d)$
\\
$B_{r,L}(x)$ \> open ball of center $x$ and radius $r$ in $(X,d_L)$
\\
$\mathcal{K}(X)$ \> space of compact subsets of $X$
\\
$d_H(A,B)$ \> Hausdorff distance of $A$, $B$ w.r.t. the distance $d$
\\
$L(X^{*},X)$ \> space of continuous and linear maps from $X^{*}$ to $X$
\\
$A_x$, $A^y$ \> $x$, $y$ section of $A \subset X \times Y$ \eqref{E:sectionxx}
\\
$\mathcal{B}$, $\mathcal{B}(X)$ \> Borel $\sigma$-algebra of $X$ Polish
\\
$\Sigma^1_1$, $\Sigma^1_1(X)$ \> the pointclass of analytic subsets of Polish space $X$, i.e.~projection of Borel sets
\\
$\Pi^1_1$ \> the pointclass of coanalytic sets, i.e.~complementary of $\Sigma^1_1$
\\
$\Sigma^1_n$, $\Pi^1_n$ \> the pointclass of projections of $\Pi^1_{n-1}$-sets, its complementary
\\
$\Delta^1_n$ \> the ambiguous class $\Sigma^1_n \cap \Pi^1_n$
\\
$\mathcal{A}$ \> $\sigma$-algebra generated by $\Sigma^{1}_{1}$
\\
$\mathcal{A}$-function \> $f : X \to \R$ such that $f^{-1}((t,+\infty])$ belongs to $\mathcal A$
\\
$h_\sharp \mu$ \> push forward of the measure $\mu$ through $h$, $h_\sharp \mu(A) = \mu(h^{-1}(A))$
\\
$\textrm{graph}(F)$ \> graph of a multifunction $F$ \eqref{E:graphF}
\\
$F^{-1}$ \> inverse image of multifunction $F$ \eqref{E:inverseF}
\\
$F_x$, $F^y$ \> sections of the multifunction $F$ \eqref{E:sectionxx}
\\
$\mathrm{Lip}_1(X)$ \> Lipschitz functions with Lipschitz constant $1$
\\
$\Gamma'$ \> transport set \eqref{E:gGamma}
\\
$G$, $G^{-1}$ \> outgoing, incoming transport ray, Definition \ref{D:Gray}
\\
$R$ \> set of transport rays \eqref{E:Rray}
\\
$\mathcal{T}$, $\mathcal{T}_e$ \> transport sets \eqref{E:TR0}
\\
$a,b : \mathcal{T}_e \to \mathcal{T}_e$ \> endpoint maps \eqref{E:endpoint0}
\\
$\mathcal S$ \> cross-section of $R \llcorner_{\mathcal T \times \mathcal T}$
\\
$g = g^+ \cup g^-$ \> ray map, Definition \ref{D:mongemap}
\\
$Jac(T)(x)$ \> Jacobian determinant $|\det(dT(x))|$, Theorem \ref{T:stimapprox} 
\end{tabbing}

\bibliography{biblio}

\end{document}